\documentclass[10pts]{amsart}

\usepackage{amsmath,amsfonts,amssymb,latexsym}
\usepackage[all]{xy}
\usepackage{hyperref}

\usepackage{times}
\usepackage{euscript} 
\def\ra{\rightarrow}
\def\cal{\mathcal} 
\def\wt{\widetilde}
\def\ol{\overline}

\def\CC{\mathbb{C}}
\def\PP{\mathbb{P}}
\def\QQ{\mathbb{Q}}
\def\ZZ{\mathbb{Z}}
\def\NN{\mathbb{N}}  
\def\RR{\mathbb{R}}

\def\OO{\cal O} 
\def\XX{\cal X}

\def\mm{\cal M}
\def\nn{\cal N}

\def\s-{\setminus}

\def\bcp{\mathbb C\mathbb P}

\newtheorem{main}{Theorem}

\newtheorem{thm}{Theorem}[section]
\newtheorem{prop}[thm]{Proposition}
\newtheorem{lemma}[thm]{Lemma}
\newtheorem{defn}[thm]{Definition}

\newtheorem{rmk}[thm]{Remark}

\numberwithin{equation}{section}

\makeatother

\begin{document}

\title{ALE Ricci-flat K\"ahler metrics and deformations of quotient surface singularities}

\footnotetext{The author is supported in part by NSF grant DMS-1007114 .}
\author{Ioana {\c S}uvaina}

\address{Ioana {\c S}uvaina, 
        Department of Mathematics,1326 Stevenson Center, Vanderbilt University, Nashville, TN, 37240}
        
\email{ioana.suvaina@vanderbilt.edu}



\date{}

\begin{abstract}
Let $N_0=\CC^2/H$ be an isolated quotient singularity with $H\subset U(2)$ a finite subgroup.
We  show that for any $\QQ-$Gorenstein  smoothings of $N_0$ a nearby fiber admits ALE Ricci-flat K\"ahler metrics in any K\"ahler class. 
Moreover, we generalize Kronheimer's results on hyperk\"ahler $4-$manifolds \cite{kron2}, by giving an explicit classification of the ALE Ricci-flat K\"ahler surfaces.

We construct ALF Ricci-flat K\"ahler metrics on the above non-simply connected manifolds. These provide new examples of ALF Ricci-flat K\"ahler $4-$manifolds, with cubic volume growth and cyclic fundamental group at infinity.
\end{abstract}

\maketitle
\tableofcontents


\section{Introduction}
\label{intro}

Orbifold K\"ahler spaces appear naturally in the study of the moduli space of metrics. In this paper we consider quotient orbifolds of the form $\CC^2/H$ with $H\subset U(2)$ a finite subgroup acting freely on $\CC^2\setminus\{0\}.$   There are two techniques to associate a smooth manifold to a singular space: one is resolving the singularity by introducing an exceptional divisor and the other is deforming the orbifold to a smooth manifold.

If we consider the resolution of the singularity, then there are two cases which behave quite differently. The first case is when $H$ is a subgroup of $SU(2).$ In this situation the singularities are called rational double points, and they are classified to be of type $A_k, D_k, k\geq1, E_6, E_7, E_8,$ corresponding to the tree configuration of self-intersection $(-2)$ rational curves which appear in the minimal resolution. We call the underlying differential manifold of the minimal resolution an $A_k, D_k, E_6, E_7$ or $E_8-$manifold. These manifolds are shown to support a hyperk\"ahler structure. They were extensively studied and classified by Eguchi-Hanson \cite{egha}, Hitchin \cite{hit}, Gibbons-Hawking \cite{giha}, Kronheimer \cite{kron1,kron2}, Joyce \cite{joy} and  others. They have proved the existence of asymptotically locally Euclidean (ALE)  Ricci-flat K\"ahler metrics on these manifolds. In the second case, when  $H$ is a subgroup of $U(2)$ rather than $SU(2),$ the minimal resolution has non-trivial canonical line bundle, and hence it does not admit Ricci-flat K\"ahler metrics. In the special case when the group $H$ is cyclic Calderbank and Singer \cite{casi} showed that the minimal resolution admits scalar flat curvature K\"ahler metrics, which are  ALE. To the author's knowledge, little is known for other groups.

The second method of obtaining a smooth manifold is by deforming the singularity. There are again two cases which arise. The first situation, when the subgroup $H$ is a subgroup of $SU(2),$ yields in fact  the hyperk\"ahler spaces, but endowed with a different complex structure. Only for $H$ a subgroup of $SU(2)$ the resolution and the smoothed manifolds are diffeomorphic.  More details on the results on hyperk\"ahler manifolds will be described in sections \ref{twistor} and \ref{clasific}.
 For $H$ a subgroup of $U(2),$ some  power of the canonical line bundle of the smoothing of the singularity is  trivial and we can prove that the manifold admits a best metric which is ALE Ricci-flat K\"ahler:


\begin{main}\label{main-A}
Let $(N_0,0)$ be an isolated quotient singularity of complex dimension $2.$ Assume there exists  $\pi: \cal{N}\to \Delta\subset \CC$ a one-parameter $\QQ-$Gorenstein smoothing  where $\Delta$ is a  disk around the origin and  $\pi^{-1}(0)=N_0.$ Then for arbitrary small $t\in \Delta\setminus\{0\}$ the manifold $N_t=\pi^{-1}(t)$ admits an ALE Ricci-flat  K\"ahler metric $g$. 
More precisely, in any K\"ahler class of $N_t$ there exists a unique ALE Ricci-flat K\"ahler metric.
\end{main}

We recall that a normal complex space is $\QQ-$Gorenstein if it is Cohen-Macaulay and a multiple of the canonical divisor is Cartier \cite[3.1]{kosb}.

In the case of a resolution, we know that the complex structure at infinity is the canonical one. When we consider deformations we need to give a more precise definition of what an ALE K\"ahler manifold is:
\begin{defn}\label{ALE}
Let $H$ be a finite subgroup of $U(2)$ acting freely on $\CC^2\setminus \{0\},$ and let $h_0$ be the Euclidean metric and $J_0$ the canonical complex structure on $\CC^2/H.$ We say that a K\"ahler manifold $(N, J, g)$ is an ALE K\"ahler manifold asymptotically to $\CC^2/H$ if there exists a compact subset $K\subset N$ and a map $\pi:N\setminus K\to \CC^2/H$ which is a diffeomorphism between $N\setminus K$ and the subset $\{z\in\CC^2/H: r(z)>R\}$  for some fixed $R\geq 0,$ such that 
$\pi_*(g)-h_0=O(r^{-4})$ and $\pi_*(J)-J_0=O(r^{-4})$.
\end{defn}

In the study of hyperk\"ahler 4-manifolds, a key ingredient used both in Hitchin's (\cite{hit}) and Kronheimer's (\cite{kron2}) proofs
was the fact that the  complex surface is birational  to a deformation of a rational double point singularity. We obtain a similar result in the Ricci-flat case:


\begin{main}\label{main-class}
Let  $(N, J, g, \omega)$  be a smooth  ALE Ricci-flat K\"ahler surface, asymptotic  to $\CC^2/H,$ with $H$ a finite subgroup of $U(2)$ acting freely on ${\CC^2}\setminus\{0\}$. The complex manifold $(N, J)$ can be obtained as the minimal resolution of a fiber of a one-parameter $\QQ-$Gorenstein deformation of the quotient singularity $\CC^2/H.$ Given any K\"ahler class $[\omega]\in H^2(N,\RR)$, then $g$ is the unique ALE Ricci-flat K\"ahler metric $g$  in that class. In particular,  if $N$ is not simply connected then its fundamental group is finite and cyclic,  its universal cover is an $A-$type manifold, and the deck-transformations are explicitly described. 
\end{main}
 
\begin{rmk}
If $N$ is simply connected then $N$ is hyperk\"ahler and the theorem is in fact a reformulation of Kronheimer's  classification result. In the case when $N$ is non simply connected, then its universal cover can not be a $D$ or $E-$type manifold.
\end{rmk}

\begin{rmk} The ALE is a necessary condition in the above classification  and it is essential for the proof of the theorem \ref{main-class}.
\end{rmk}

In section  \ref{gib-hawk}, we show that the above non-simply connected manifolds  also admit complete asymptotically locally flat (ALF) Ricci-flat K\"ahler metrics and we also give examples of non-simply connected manifolds  with infinite topology which have  complete Ricci-flat K\"ahler metrics, by emulating a  construction of Anderson, Kronheimer and LeBrun \cite{akl}.

If one would like to generalize the $2-$dimensional results to higher dimensions, then there are several approaches which are used for obtaining good metrics on resolutions of singularities. There is no analog of the deformation part, as Schlessinger Rigidity Theorem says that quotient  singularities $\CC^n/H$ with singularity of codimension three or more have no non-trivial deformations \cite{schl}.

\section{Smoothings of  quotient singularities}\label{def-thm}

In this section we give a short summary of the Koll\'ar and Shepherd-Barron's results \cite{kosb} on the classification of isolated quotient  singularities which admit a smoothing.
\begin{defn}
A flat, surjective map $\pi:\XX\ra \Delta,$ where $\Delta\subset\CC$ is an open neighborhood of $0,$ 
 is called a one-parameter $\QQ-$Gorenstein smoothing of a
 normal variety  $X_0,$ if $\pi^{-1}(0)=X_0,$ and the following conditions are satisfied. 
\begin{itemize}
\item [ i)] $\XX$ is $\QQ-$Gorenstein
\item [ ii)] $X_t=\pi^{-1}(t)$ is smooth for every $t\in \Delta-\{0\}.$
\end{itemize}
\end{defn}

For a definition of a $\QQ-$Gorenstein variety and more details on the algebraic geometry aspects see \cite[Def. 3.1]{kosb}. In their paper, Koll\'ar and Shepherd-Barron prove the following classification theorem:

\begin{thm}\cite[3.10]{kosb}\label{deform}
An isolated quotient   surface singularity $(N_0,0)$ which admits a one-parameter $\QQ-$Gorenstein smoothing is either a rational double point or a cyclic singularity of type $\frac1{dn^2}(1,dnm-1)$ for $d>0, n\geq 2$ and $n,m$ relatively prime.
\end{thm}

From a topological point of view the two types of singularities are distinguished by the fundamental group of a smoothing. The smoothing of a rational double point is simply connected, while the deformations of the second type have finite cyclic fundamental group.
On the algebraic geometric side of the picture this is reflected on the canonical line bundle being trivial, or non-trivial torsion, respectively.

As the rational double point singularities and their preferred metrics are well understood \cite{hit,kron1}, we focus our attention on the second type of orbifolds.
A singularity of type $\frac1{dn^2}(1,dnm-1)$ is a quotient singularity $\CC^2/\Gamma_{dn^2},$ where $\Gamma_{dn^2}=\{\rho | \rho^{dn^2}=1\}$ is the finite cyclic group of order $dn^2$ which acts on $\CC^2$ diagonally with  weights $(1, dnm-1),$ i.e. if $\rho\in \Gamma_{dn^2}$ and $(z_1,z_2)\in\CC^2$ then $\rho(z_1,z_2)=(\rho z_1, \rho^{dnm-1}z_2).$
If we consider the subgroup of order $dn,$  $\Gamma_{dn}\subset\Gamma_{dn^2},$ then its quotient singularity is of type $\frac1{dn}(1,-1),$ which is the $A_{dn-1}$ rational double point singularity. The quotient singularity $\frac1{dn}(1,-1)$ can be biholomorphically embedded as a hypersurface in $\CC^3,$ $\CC^2/\Gamma_{dn}\to M_0=(xy=z^{dn})\subset\CC^3,$ by the following map: $(z_1,z_2)\to (x,y,z)=(z_1^{dn}, z_2^{dn}, z_1z_2)\in\CC^3.$ On $M_0$ there is an induced action of the quotient group $\Gamma_{dn^2}/\Gamma_{dn}.$ As this group is cyclic of order $n$, we can represent its action by the action of a unit complex number $\xi$ of order $n, \xi^n=1,$ as follows: $\xi(x,y,z)=(\xi x, \xi^{-1}y, \xi^mz).$

The deformations of $A-$type singularities $M_0$ are known. They are given by $\mm\subset \CC^3\times\CC^{dn}, ~
\mm= (xy=z^{dn}+e_1z^{{dn}-1}+\dots+e_{dn}),$ where $(e_1,\dots,e_{dn})\in\CC^{dn}.$ The action of the group $\Gamma_{dn^2}/\Gamma_{dn}$ can be extended trivially on the coordinates of $\CC^{dn}:$
$$
\xi(x,y,z,e_n,\dots ,e_{dn})= (\xi x,\xi ^{-1}y,\xi ^{m}z,e_n,\dots ,e_{dn})
$$
Then, there is a special family of deformations  $\mm'\subset\CC^3\times\CC^{d}$ which is $(\Gamma_{dn^2}/\Gamma_{dn})-$invariant and it is given by the  equation 
$$
xy=z^{dn}+\sum_{j=1}^{d}e_{jn} z^{(d-j)n},
$$ 
where $e_{n},\dots ,e_{dn}$ are linear coordinates on $\CC^d.$     Let $\nn=\mm'{/(\Gamma_{dn^2}/\Gamma_{dn})}$ and $\phi:\nn\ra \CC^d$ be the quotient map of the projection $\mm' \ra \CC^d.$

\begin{prop}\cite{kosb}\label{unsm} 
The map $\phi:\nn\ra \CC^d$ is a $\QQ-$Gorenstein deformation of the cyclic singularity of type  $\displaystyle \frac{1}{dn^2}(1,dnm-1).$
Moreover, every one-parameter $\QQ-$Gorenstein deformation $\XX \ra \Delta$ of a singularity of type 
$\displaystyle \frac{1}{dn^2}(1,dnm-1)$ is isomorphic to the pullback through $\phi$ of some germ of a 
holomorphic map $(\Delta,0)\ra(\CC ^d,0).$ 
\end{prop}

If we consider the case of one-parameter $\QQ-$Gorenstein deformations which are equivariant under some free action of the cyclic group of order $n$, $\ZZ_n,$ and such that the fibers have isolated singularities, then we can ask if we can obtain a similar classification result.
This situation was studied by Manetti in \cite{man} by using Catanese's  explicit classification \cite{cat} of the automorphism group of rational double point singularities. He proves the following:

\begin{prop}\label{defman}
Let $(N_0, 0)\subset(\CC^3,0)$ be a rational double point singularity and let $\mu_n$ be an action of the finite cyclic group of order $n$ on $(N_0, 0)$ coming from a subgroup of  $GL(3, \CC).$ Moreover, there exists a one-parameter deformation of $(N_0, 0),$ $(\mathcal{N}, 0)\subset (\CC^3\times\Delta)\to\Delta$  which is $\mu_n$ invariant, where the action extends trivially to the $\Delta$ coordinate. Assume further that the $\mu_n$ acts freely on $\mathcal{N}\setminus\{0\}.$ Then the singularity $(N_0, 0)$ is a singularity of the form $\displaystyle \frac{1}{dn^2}(1, dnm-1),$ where $n, m$ are relatively prime integers, and the deformation family is of the form described in Proposition \ref{unsm}.
\end{prop}



\section{ALE Ricci-flat K\"ahler metrics via the twistor space}\label{twistor}


On the minimal resolution of a rational double point singularity, ALE hyperk\"ahler  metrics (in particular Ricci-flat K\"ahler metrics) where constructed by Hitchin \cite{hit}, Kronheimer  \cite{kron1,kron2}, and others. In this section we show how their results can be rewritten in the set-up of Theorem \ref{main-class} and we analyze in detail the deformations of the second type of singularities from Theorem \ref{deform}. As we observed in Proposition \ref{unsm}, these singularities and their deformations are closely related to the $A-$type singularities. We start first by reviewing Hitchin's construction of hyperk\"ahler metrics on deformations of these orbifolds. This is done via the twistor space construction. 

The Penrose twistor correspondence \cite{ahs} is a one-to-one correspondence between half-conformally flat $4-$manifolds and complex $3-$folds, the twistor spaces, which satisfy certain properties. 
We will give a short summary of the twistor space of a hyperk\"ahler manifold $M$, by following Hitchin's exposition in \cite{hit}.  The hyperk\"ahler metrics are anti-self-dual if $M$ is endowed with the orientation induced by any of the complex structures.
In our particular case, we have the following version of the Penrose correspondence between  a simply connected manifold $M$ endowed with a hyperk\"ahler metric $g$ and the complex $3-$fold $Z$ which has the following properties:
\begin{enumerate}
\item   $Z$ is the total space of a holomorphic fibration $\pi:Z\to \bcp^1. $
\item   There exists a $4-$parameter family of sections of $\pi,$ each with normal bundle $\OO(1)\oplus \OO(1).$
\item	 There exists a non-vanishing holomorphic section $s$ of $K_Z\otimes\pi^*\OO(4)$.
\item There exists a real structure, i.e. a free anti-holomorphic involution, on $Z$ which is the extension of the antipodal map on $\bcp^1$, such that $\pi$ and $s$ are real, and $Z$ is fibered by the real sections of the family.
\end{enumerate}
  
  The space of real sections recovers the manifold $M$ and the conformal class of the metric. We use property $(3)$ to fix a metric in the conformal class.  Property $(1)$ is specific for hyperk\"ahler manifolds.
  For each $[u_1:u_2]\in \bcp^1,$ the projection from $Z$ to the space of real sections identifies $\pi^{-1}([u_1:u_2])$ with $M,$ which induces a complex structure on $M$ such that the metric $g$ is K\"ahler. 

  We give now a short account of the twistor space associated to a smoothing of an $A_{k-1}-$type singularity with a Ricci-flat K\"ahler metric, by following Hitchin's presentation in \cite{hit}. 
  A smoothing of an $A_{k-1}-$singularity is biholomorphic to $M=(xy=z^k+e_1z^{k-1}+\dots+e_k),$ for some constants $e_i.$ In order to construct $Z$ we need a family of complex structures on $M$, parametrized by $\bcp^1.$ Hitchin \cite{hit} considers the manifold $\tilde{Z}$
\begin{equation}\label{twist}
  xy=z^k+e_1(u)z^{k-1}+\dots+e_k(u)
\end{equation}
  as a first approximation,
  where $u$ is an affine coordinate and $e_i(u)$ are holomorphic functions. As we need to define a $\bcp^1-$family of complex structures, we need to consider the extension of the coordinate $u$ to homogeneous coordinates $[u_1:u_2]$ on $\bcp^1,$ then  $e_i(u)$ are going to be extended to sections of holomorphic line bundles on $\bcp^1,$ while $x,y,z$ are local coordinates on certain line bundles. Using the properties $(1-4)$ Hitchin  shows that $x,y,z$ need to be local coordinates on $\OO(k), \OO(k), \OO(2)$, respectively, while each $e_i$ is a section of $\OO(2i).$
  Hence $\tilde{Z}$ is a hypersurface in the vector bundle 
  $$\tilde{Z}\subset \OO(k)\oplus\OO(k)\oplus\OO(2)\to \bcp^1.$$ 
  The hypersurface $\tilde{Z}$ has at most rational double point singularities, which can be resolved \cite{hit} to obtain a manifold $Z$ which is the twistor space of $M$. Moreover, all the structures defined on $\tilde{Z}$ (properties 1-4) lift to corresponding structures on $Z.$ 

   On  $\tilde Z,$ Hitchin constructs a natural  real structure $\tau$ induced from a real structure on $\OO(k)\oplus\OO(k)\oplus\OO(2)\to \bcp^1$ (property $4$, and a non-vanishing holomorphic section $s$ (property 3).   
The last ingredient missing is a $4-$parameter family of sections. First, we make the assumption, \cite{hit}, that the equation \ref{twist}  factorizes as:
\begin{equation}\label{tw-fact}
 xy=\prod_{i=1}^k(z-p_i)
\end{equation}
where $p_i\in\Gamma(\OO(2)).$ 
As a function on the affine parameter $u=\frac{u_1}{u_2}\in U\equiv\CC\subset\bcp^1$ the polynomials $p_i$   can be written as   $p_i(u)=a_i u^2+2b_i u+c_i.$
 Let $(x(u), y(u), z(u))$ be a section of $\pi:Z\to \bcp^1,$ and let $z(u))=a u^2+2b u+c.$ 
 If we require that $Z$ and $(x(u), y(u), z(u))$ are  $\tau-$invariant, then we need to consider $p_i(u)$ satisfying $p_i(\tau u)=-\tau (p_i(u))$, or equivalently $c_i=-\bar{a}_i, b_i=\bar{b}_i, $ and  $c=-\bar{a}, b=\bar{b}.$ We can solve the equation \ref{tw-fact} by a simple factorization:
 $$x(u)=A\prod_{i=1}^k (u-\frac{(b-b_i)-\Delta_i}{(a-a_i)}), y(u)=B\prod_{i=1}^k(u-\frac{(b-b_i)+\Delta_i}{(a-a_i)})$$
 where $\Delta_i=\sqrt{(b-b_i)^2+|a-a_i|^2}$ is the discriminant of $z-p_i$ and $A,B$ are constant coefficients which must satisfy $AB=\prod (a-a_i).$ The reality condition on the polynomials $x(u)$ and  $y(u)$ implies that 
 \begin{equation}\label{ang}
 A\bar{A}=\prod_{i=1}^k((b-b_i)+\Delta_i).
  \end{equation}
  
 By Penrose correspondence the subspace of real sections $M_{re}$, can be identified with $M$ endowed with a conformal structure.  As we saw, the real sections are parametrized by $(b, Im (a), Re (a))\in \RR^3$ and the angular coordinate associated to $A$. Moreover, each conformal class is determined by a choice of the real quadratic polynomials $p_i$, or equivalently $k$ points $x_i=(b_i, Im(a_i), Re(a_i))$ in $\RR^3.$ 
   On this real submanifold, the conformal structure is given by \cite[eq.4.4]{hit}:
   \begin{equation}\label{conf-met}
   \gamma^2(db^2+dad\bar{a})+\big(Im(\frac{2dA}A -\delta da)\big)^2=
   \gamma^2dad\overline{a}+(\frac{2dA}A-\delta da)(\frac{2d\overline{A}}{\overline{A}}-\overline{\delta} d\overline{a})
   \end{equation}
where 
   \begin{align}
 & \gamma =\sum_{i=1}^k((b-b_i)^2+|a-a_i|^2)^{-\frac12}=\sum \Delta_i^{-1}\notag \\
  &\delta= \sum_{i=1}^k \frac{-(b-b_i)+\Delta_i}{\Delta_i(a-a_i)} \notag
  \end{align}

  To recover the metric on $M$, first we assume that $(M, J)$ is identified with the smooth hypersurface corresponding to 
  $\pi^{-1}(0),$ i.e. $M=(xy-\prod(z+\bar{a}_i))$ where $a_i\neq a_j,$ for $i\neq j.$ 
  We have  local coordinates $(y,z)$ whenever $y\neq0.$ 
  In terms of these local coordinates we have the following identifications:
  \begin{equation}\label{local}
  u=0,~~z=-\bar{a},~~y=\frac1A\prod_{i=1}^k((b-b_i)+\Delta_i)=\bar{A}
  \end{equation}
  In these local coordinates, the metric corresponding to the twistor space data is the following \cite{hit}:
  \begin{equation}\label{metric}
 g= \gamma dzd\bar{z}+\gamma^{-1}(\frac{2dy}y+\bar{\delta}dz)(\frac{2d\bar{y}}{\bar{y}}+\delta d\bar{z})
  \end{equation}
  where
  \begin{align}
 & \gamma =\sum_{i=1}^k((b-b_i)^2+|\bar{z}+a_i|^2)^{-\frac12}=\sum_{i=1}^k \Delta_i^{-1}\notag \\
  &\delta= \sum_{i=1}^k \frac{(b-b_i)-\Delta_i}{\Delta_i(\bar{z}+a_i)} \notag
  \end{align}  
  and $b$ is defined implicitly by
  \begin{equation}\label{implic}
  \displaystyle \prod_{i=1}^{k} ((b-b_i)+\Delta_i)=y\bar{y}.
  \end{equation}
  
 An easy computation shows that the metric extends smoothly at $y=0,$ and that 
 the above metric is asymptotically locally Euclidean and complete.
 
  The K\"ahler form associated to the metric is given by: 
  $$\omega_g=i \gamma dz\wedge d\bar{z}+i\gamma^{-1}(\frac{2dy}y+\bar{\delta}dz)\wedge(\frac{2d\bar{y}}{\bar{y}}+\delta d\bar{z}).$$
  
   Moreover, Hitchin shows that    we can choose a distinguished basis $\gamma_1, \dots,\gamma_{k-1}$, of $H_2(M, \ZZ) $ and then the K\"ahler class of $\omega$ is given by the Poincar\'e dual of $\sum8\pi(b_i-b_{i+1})\gamma_i\in H_2(M,\RR)$.

The condition that $a_i\neq a_j$ for $i\neq j,$ is equivalent to the fact the complex surface $(M,J)=\pi^{-1}(0)$ is smooth. If this condition is not satisfied then both $M$ and $\tilde Z$ have singularities above $0$, which Hitchin shows that can be resolved by introducing exceptional divisors. If the corresponding $b_i$'s are distinct then there are corresponding second homology classes $\gamma_i$ which can be represented by exceptional $(-2)-$ rational curves. Equal corresponding $b_i$'s will yield an orbifold ALE Ricci-flat K\"ahler space $(M, J, g).$

We  conclude that the metric $g$ is uniquely determined by the complex constants $a_i$ which are given by the complex structure $(M,J)$ and the real numbers $(b_i-b_{i+1}), i=1,\dots,k-1,$ which determine the K\"ahler class of $\omega_g.$ The Euclidean translation along the $b-$axis induces equivalent metrics. We have the following reformulation of Hitchin's result:

\begin{thm}[Hitchin \cite{hit}]\label{ex-A}
Let $(M_0, 0)$ be a quotient singularity of type $A_{k-1}.$ Let $(M_t, J_t)$ be the fiber of a $\QQ-$Gorenstein deformation of $(M_0,0),$ and  $(M,J) $ its minimal resolution.
For any such $(M, J)$   and any K\"ahler class $\Omega\in H^2(M,\RR)$ there exists a unique ALE Ricci-flat K\"ahler metric $g$ such that its K\"ahler form $\omega_g$ is in the class $\Omega.$
\end{thm}

\begin{rmk}\label{exist}
We have argued, by giving full details, that Theorem \ref{main-A} is true for $A-$type singularities, but if we consider Kronheimer's results in \cite{kron1, kron2} then the above statement is true for any rational double point singularities.
\end{rmk}

We return to  the case when $k=dn$ and consider the singular manifold $M_0=(xy=z^{dn})$ and a smoothing of the form $M=M_{d,n}=(xy=z^{dn}-\sum_{i=1}^d e_i z^{n(d-i)}),  e_i\in\CC.$ This deformation is $\ZZ_n\cong\Gamma_{dn^2}/\Gamma_{dn}-$invariant,  where the group $\ZZ_n\cong\{\rho | \rho\in\CC,~ \rho^n=1\}$ acts by:
\begin{equation}\label{act}
\rho (x,y,z)=(\rho  x,\rho^{-1}y, \rho^m z), ~\text{where} ~(m,n)=1
\end{equation}
We would like to find which of the Ricci-flat K\"ahler metrics  constructed above  on $M$ are also $\ZZ_n-$invariant, and for this
we need to extend the action to the twistor space associated to $M.$
 
 First, we  study the special case of $d=1,$ and later discuss the situation of arbitrary $d.$
For $d=1,$ after a change in coordinates, we can assume $M$ to be given by the equation $(xy=z^n-1).$ This equation can be factored as $xy= \prod_{i=1}^n(z-\rho^{i}),$ where $\rho=e^{\sqrt{-1}\frac{2\pi }n}.$ Hence, if we are looking for the twistor space associated to $M,$ where the fixed complex structure is obtained at $u=0$, then we are looking for a twistor space which has a first approximation $\tilde{Z}\subset \OO_{\bcp^1}(n)\oplus\OO_{\bcp^1}(n)\oplus\OO_{\bcp^1}(2),$ defined by $xy=\displaystyle\prod_{i=1}^n(z-p_i(u)),$ where $p_i(u)=-\rho^{-i} u^2+2b_iu+\rho^{i}.$ Next we need to extend the $\ZZ_n-$action \ref{act} to $\OO(n)\oplus\OO(n)\oplus\OO(2)\to \bcp^1,$ such that $\tilde{Z}$ is invariant.
Let $\rho *u$ denote the action on the affine coordinate. The equation  $xy= \prod_{i=1}^n(z-p_i(u)),$ becomes under the action of $\rho\in\ZZ_n:$ 
$$\rho x\rho^{-1}y= \prod_{i=1}^n(\rho^m z-(-\rho^{-i}(\rho*u)^2+2b_i(\rho*u)+\rho^{i})),$$ 
or  equivalently
$$xy=\prod_{i=1}^n\big(z-(-\rho^{-i-m}(\rho*u)^2+2b_i\rho^{-m}(\rho*u)+\rho^{i-m})\big).$$
Requiring $\tilde{Z}$ to be invariant implies $p_{i-m}(u)=-\rho^{-i-m}(\rho*u)^2+2b_i\rho^{-m}(\rho*u)+\rho^{i-m},$ which gives the action on $u, \rho*u=\rho^m u$ and also $b_i=b_{i-m}.$ As $m$ and $n$ are relatively prime, we have that $b_1=b_2=\dots=b_n,$ and we can assume $b_i=0.$
This fixes the position of the points on  $\RR^3,$ as vertices of a regular polygon with $n$ sides,  in the plane $b=0$ centered at the origin. Moreover, using the identification \ref{local} of the local coordinates, we can identify the exact $\ZZ _n$ action on $M_{re}$ as follows:
\begin{equation}\label{act-m^c}
\rho(a)=\rho^{-m}a,~ \rho (A)=\rho A, ~ \text{and} ~\rho (b)=b
\end{equation}
To identify the action on the $b-$coordinate, we only need to observe that the defining implicit equation \ref{implic} is invariant under the action, and as $b$ is a real number, the action is trivial.

An easy computation shows that $\gamma$ is invariant under the $\ZZ_n-$action, while $\rho (\delta)=\rho^m \delta.$
Hence  $g$ is invariant under the free action of $\ZZ_n$ and induces a smooth, complete  ALE Ricci-flat  K\"ahler metric on $
N=N_{d,n}=M/\ZZ_n.$ The quotient metric is locally hyperk\"ahler, but not globally hyperk\"ahler, as $(N, g)$ is a deformation of the manifold $(N_0=\CC^2/\Gamma_{dn^2},\hat{g}_0),$ where $\Gamma_{dn^2}\subset U(2),$ but $\Gamma_{dn^2}$ is not a subgroup of $ SU(2),$ and  $\hat {g}_0$ is the canonical quotient metric on $\CC^2/\Gamma_{dn^2}.$  In terms of our construction of twistor spaces, this translates in the fact that the group  acts nontrivially on $u\in\bcp^1$ (the complex structures on $M$) as a rotation by a factor of $\rho^m$  fixing the initial complex structure $J,$ as well as $-J.$ 

In the case $d=1,$ the quotient manifold $N$ has Euler characteristic equal to $\frac {\chi(M)} n=\frac{(1+(n-1))} n=1$ and $N$ is a rational homology ball with infinity asymptotic to $\CC^2/\Gamma_{n^2}.$ Moreover, proposition \ref{unsm} tells us that there exists a unique smoothing $(N, J)$ up to biholomorphisms. We have showed:

\begin{prop}
The rational homology ball $(N, J),$ with fundamental group $\ZZ_n$ and infinity of the form $L(n^2, nm-1)\times(R,\infty),$ constructed above admits a unique ALE Ricci-flat K\"ahler metric.
\end{prop}

Given an arbitrary $d,$ for the defining polynomial of $M,$ 
$xy=z^{dn}-\displaystyle\sum_{i=1}^d d_i z^{n(d-i)},$ we can consider the following factorization  $ xy=\displaystyle\prod_{i=1}^d (z^n-d'_i)=\prod_{i=1}^d(z^n-c_i^n)$ where $c_i^n=d'_i$ and $c_i$ has the smallest possible argument. 
As we assume that $M$ is smooth, then the sets $\{c_i^n=d'_i\}_i$ and  $\{c_i\}_i$ contain distinct numbers. The defining polynomial decomposes furthermore as:
\begin{equation}\label{gendef}
 xy=\prod_{i=1}^d\prod_{j=1}^n(z-c_i\rho^{j})
 \end{equation}
In particular, we have $a_{i,j}=-\bar c_i\rho^{-j}, i=1,\dots, d, j=1,\dots n,$.

The same arguments, as in the case $d=1,$ give us the position of the $dn$ points $(b_{i,j}, Im(a_{i,j}), Re(a_{i,j}))$ 
 in $\RR^3,$  as the vertices  of $d$ regular polygons with $n$ sides, each polygon in a horizontal plane  $b_{i,j}=b_i=$constant, $i=1,\dots, d$, centered at the origin of the plane $(b_i, 0, 0)$ and with one of the vertices at $(b_i,-\bar c_i)\in\RR\times\CC\cong\RR^3.$ 
 The $\ZZ_n-$actions on $\tilde Z$ and $M_{re}$ are the same as in the simple case \ref{act-m^c}. 
 
The group action rearranges the terms in the formula of $\gamma$ and after factorizing out $\rho^m$ we obtain a similar formula for $\delta,$ i.e. $\gamma$ is invariant under the $\ZZ_n$ action, while $\rho (\delta)=\rho^{m} \delta.$ Hence, the Hitchin's metrics, given by the equation \ref{metric},   are $\ZZ_n-$invariant for the above choice of points. The fact that the $b-$coordinates of the vertices satisfy the conditions $b_{i,j}=b_i$ translates to the fact that the K\"ahler class of the metric is $\ZZ_n-$invariant. Moreover any invariant K\"ahler class must satisfy these conditions.

We are now ready to give the proof of the main theorem:
\begin{proof}[Proof ~of~Theorem \ref{main-A}]
From Theorem \ref{deform} we know that there are only two types of isolated quotient singularities which admit a $\QQ-$Gorenstein smoothing. For the first type of singularities the 
existence of ALE Ricci-flat  K\"ahler metric in any K\"ahler class was proved by Hitchin \cite{hit}, for the $A-$singularities, and by Kronheimer \cite{kron1,kron2} in general (according to the reformulation of their results in Theorem \ref{ex-A} and Remark \ref{exist}). 

In the case of a $\frac1{dn^2}(1, dnm-1)$ singularity, we saw that a smoothing $N_t$ of $\CC^2/\Gamma_{dn^2}$ is the  $\ZZ_n-$quotient of a $A_{dn-1}$ manifold, see Proposition \ref{unsm}. In this section we looked at the explicit description of Hitchin's metrics on these deformations and proved that the K\"ahler metric is also group invariant as long as it satisfies the condition that its K\"ahler class is invariant. 

As any K\"ahler class on $N$ can be pulled back to a $\ZZ_n-$invariant  K\"ahler class on a $A_{dn-1}-$manifold, which contains a unique ALE Ricci-flat  K\"ahler metric \cite{hit} (\ref{ex-A}), which we proved is group invariant, hence it induces the needed metric on $N$. 
\end{proof}

\section{Classification of ALE Ricci-flat  K\"ahler metrics on surfaces}\label{clasific}


The simply connected ALE Ricci-flat  K\"ahler surfaces are classified by Kronheimer \cite{kron2}. 
These are the ALE hyperk\"ahler  surfaces, and their metrics are explicitly described. 
When the manifolds are not simply connected, the author could not find such a classification in the literature. The present paper is intended to fill in this gap, by giving a complete description of the ALE  Ricci-flat  K\"ahler, with emphasis on the non-simply connected case. 

In this section we are going to use the Kronheimer's approach to classify the ALE Ricci-flat  K\"ahler surfaces $(M,J,g)$. We will see that the 4-manifold $M$ and the metric $g$ are determined by the  topology of the asymptotic end and the K\"ahler class of the metric, respectively. We prove this by starting with some easy lemmas.

\begin{lemma}\label{finite-fgr}
Let $(M,g)$ be an ALE Ricci-flat  manifold then the fundamental group of $M$ is finite.
\end{lemma}
\begin{proof}
This  is an immediate consequence of the fact that $(M,g)$ is  asymptotically locally Euclidean  and the Volume Comparison Theorem on Ricci-flat manifolds. The Volume Comparison Theorem for the universal cover $(\wt M,\wt g)$ tells us that $(\wt M,\wt g)$ must have volume growth less than that of $\RR^4,$ and as $M$ is ALE we have that the universal cover must be a finite covering of $M$. Hence $\pi_1(M)$ is finite.
\end{proof}

\begin{lemma}\label{1-end}
Let $(M,g)$ be an ALE Ricci-flat $4-$manifold, then its  universal cover $\widetilde{M} $  has only one end, in particular is ALE.
\end{lemma}

\begin{proof}
If $\wt M$ is disconnected at infinity, then $(\wt M, \wt g)$ contains a line (\cite[Ch.9]{pet}), and as $Ric_{\wt g}\equiv0$ the Splitting Theorem tells us that  $(\wt M, \wt g)$ is isometric to a product $(N\times \RR, g_0+dt^2).$ In particular the $3-$dimensional manifold $N$ is Ricci-flat, and hence flat. But as $N$ is simply connected, we have $N$ is the Euclidean space, and so is $\wt M.$ In particular, $\wt M$ has only one end at infinity. 
\end{proof}

\begin{rmk}
 The lemma actually proves a stronger result. 
If the asymptotic  models of $M$ and $\wt M$ are those of $\RR^4/H$ and $\RR^4/\Gamma,$ with $H, \Gamma $  finite subgroups of $SO(4),$  and we denote the fundamental group $\pi_1(M)$ by $G,$ then $H$ is an extension of $\Gamma$ by $G:$
 $$0\to\Gamma\to H\to G\to0.$$
\end{rmk}
\begin{prop}\label{equiv-cl}
The universal cover $(\wt M, \wt J, \wt g)$ of an ALE Ricci-flat K\"ahler  manifold $( M, J, g)$ is an ALE hyperk\"ahler manifold, and $M$ is obtained by taking the quotient of $\wt M$ by a finite cyclic group of automorphisms.
\end{prop}

\begin{proof}
Lemma \ref{1-end} tells us that the universal cover $(\wt M, \wt g)$ is ALE Ricci-flat K\"ahler manifold, and as $\wt M$ is simply connected, it is hyperk\"ahler \cite{hklr}. In particular, as the metrics are compatible with the complex structure, the asymptotics of $M, \wt M$ are given by $H, \Gamma$ finite subgroups of $U(2), SU(2),$ respectively.

Let $\Gamma'$ be the normal subgroup of $H$ given by $\Gamma'=H\cap SU(2).$ Then $\Gamma \subseteq\Gamma'.$ First we will show that the two subgroups are equal to each other. 
As $\Gamma'/\Gamma$ is a subgroup of $H/\Gamma=G=\pi_1(M)$, it has an associated covering 
$(M', g')\to (M, g),$ with $g'$ the pull-back metric.  As $\wt M$ is the universal cover of $M,$ we  have a covering 
$(\wt M, \wt g)\to (M', g').$ Moreover, $M'$  has one end at infinity and  the group of deck transformations
 is given by $\pi_1(M')=\Gamma'/\Gamma.$ As $\wt M$ is a hyperk\"ahker manifold, the manifold $M'$ will  be hyperk\"ahler if $\pi_1(M')$ acts trivially on the canonical bundle $K_{\wt M}.$ 
 But any element in $\pi_1(M')$ can be identified with an element in $\Gamma'/\Gamma,$ 
 and can be represented be a loop $\gamma$ in the asymptotic model, $\CC^2/\Gamma'.$ 
 As $\Gamma' \subset SU(2)$ the loop acts trivially on $K_{\wt M}.$ Hence $M'$ is also a hyperk\"ahler manifold. 
 But Kronheimer \cite{kron2} classified all ALE hyperk\"ahler manifolds, and showed that they are, in particular, simply connected. Hence $\wt M =M'$ and $\Gamma=\Gamma'.$ 

This implies that $H/\Gamma, (\Gamma=H\cap SU(2))$ is a subgroup of $U(2)/SU(2)=S^1.$ So, $G=H/\Gamma $ is a finite cyclic group.

 \end{proof}

\begin{rmk} In particular, if the manifold $M$ is not simply connected then it is not hyperk\"ahler, just locally hyperk\"ahler, as the group $G$ acts non-trivially on the canonical bundle $K_{\wt M}.$ Moreover, $M$ is a spin manifold iff the order of  $G$ is odd.
\end{rmk}

By Proposition \ref{equiv-cl}, classifying the non-simply connected ALE Ricci-flat K\"ahler  manifolds is equivalent to classifying the  finite cyclic {\it free} groups of isometries on  hyperk\"ahler $4-$manifolds. 


 In the remaining part of this section, we  prove that the manifold $\wt M$ admits a $G-$equivariant degeneration to the quotient orbifold $\CC^2/\Gamma.$ This will lead to an explicit description of $M$ and it will give us a complete classification of the ALE Ricci-flat K\"ahler manifolds. 
To do this we go back to Kronheimer's paper \cite{kron2} and 
study the $G-$action on the twistor space $Z_{\tilde M}.$

As we mentioned in the previous section, to any half-conformally flat (spin) manifold $(X, [g])$ we can associate a complex three-manifold $Z,$ its twistor space. If the manifold $(X, g)$ is scalar flat K\"ahler, then the metric is anti-self-dual (ASD), and its twistor space is the projective bundle $Z=\PP(V^+),$ where $V^+$ is the bundle of self-dual spinors. If we consider the manifold with the reversed orientation $\overline X$, then $(\overline X, g)$  is a self-dual manifold, and to it we associate the same twistor space $Z,$ which now is defined as $\PP(V^-), $ where $V^-$ is the bundle of anti-self-dual spinors on $(\ol X, g).$

Given an ALE Ricci-flat hyperk\"ahler manifold $(\wt M, \wt g),$ Kronheimer \cite{kron2} constructs  orbifold compactifications of both $\wt M$ and $Z_{\wt M} as follows:$ $\ol {\wt M}=\wt M\cup\{\infty\}$ and $ Z_{\ol{\wt M}}=Z_{\wt M}\cup\{\bcp^1/\Gamma\}$. Let $U'\cong(\RR^4\setminus B_r(0))/\Gamma$ be a neighborhood of infinity in $\wt M,$ such that $U'/G$ is a neighborhood of infinity in $M$ and let $p:\wt U'\to U'$ be the universal cover of $U'$. If $x_i$ are $H-$invariant local coordinates at infinity in $\wt U',$ then $\wt U=\wt U'\cup\{\infty\}$ is a smooth chart if we consider the coordinates $y_i=x_i/r^2.$ Moreover, if $\phi:\wt M\to \RR^+$ is a smooth $G-$invariant function, equal to $1/r^2$ outside some compact set, then the metric $\ol g=p^*(\phi^2\wt g|_{ U'})$ extends to a   class C$^3$ metric on $\wt U$ \cite{kron2}, which is $H-$invariant. Then $(\wt U,\ol g)$ defines Riemannian orbifold structures on the compactifications $\ol {\wt M}=\wt M\cup\{\infty\}$ and $ \ol M=M\cup\{\infty\}.$  The manifold $\ol M$ is obtained from $\ol {\wt M}$ by taking the quotient by the group $G,$ which acts freely on $\wt M$ and fixes $\infty$ on $\ol {\wt M}$.

We now construct a compactification of the twistor space $Z$ by defining the twistor space for an orbifold space.  As $\ol g$ is in the same conformal class as $\wt g,$ it is  ASD in $x_i$ coordinates. But if we consider $y_i$ coordinates, as the transition map is reversing the orientation, the metric is  be self-dual. To  leading order, $H$ acts linearly on both $x_i$ and $y_i$ coordinates, by the same representation $\rho_H:H\to U(2)$, and so does its subgroup $\Gamma.$ Moreover (see \cite[Lemma 2.1]{kron2}), in $y_i$ coordinates  the action of $\Gamma $ on $V^+$ is trivial, but non-trivial on $V^-.$  For the neighborhood $(\wt U, \ol g) $ of $\infty$ we consider the twistor space $Z_{\wt U}.$ As the metric $\ol g$ is both $H$ and $ \Gamma-$invariant, the actions of $H$ and $\Gamma$ extend to holomorphic actions on $Z_{\wt U}.$ Moreover, if $\wt l_\infty\subset Z_{\wt U}$ is the twistor line corresponding to the $\infty,$ then $H$ and $\Gamma$ act on $\wt l_\infty$ by the standard action of $U(2)$ on $\bcp^1$ \cite{kron2}. Let $l_\infty=\wt l_\infty/\Gamma.$ We define the compactifications of the twistor spaces to be $ Z_{\ol{\wt M}}=Z_{\wt M}\cup l_\infty$ and $Z_{\ol M}=Z_M\cup (\wt l_\infty/H),$ with the complex structures given by $\Gamma, H-$quotients of $Z_{\wt U}$ in the neighborhoods of the twistor lines at infinity. The complex $3-$fold $Z_{\ol M}$ is obtained as  the quotient $Z_{\ol{\wt M}}/G,$ where $G$ acts freely on $Z_{\wt M}$ and non-trivially on $ l_\infty.$

As $\wt M$ is  hyperk\"ahler, there is a holomorphic projection $\pi:Z_{\wt M} \to \bcp^1.$    The sphere $\bcp^1$ gives a $S^2-$family of  complex structures on $\wt M:$ to each point $a\in\bcp^1$ there is an associated complex structure $J_a$ on $\wt M$ given by the complex structure on the fiber $\pi^{-1}(a)\cong\wt M\hookrightarrow Z_{\wt M}.$ Note that the compactification $\ol{\wt M}$ is not hyperk\"ahler and  the projection $\pi$ does not extend to the compactification of the twistor space. 
\vspace{.1in}

{\bf Kronheimer's  description of the twistor space of $\wt M:$}\\
On  $Z_{\wt M}$ there is a  preferred line bundle $\pi^*\OO(1).$ By Hartog's theorem this extends to a line bundle on $Z_{\ol{\wt M}},$ which can be restricted to a line bundle on $l_\infty.$
To understand the twistor space $Z_{\wt M},$ Kronheimer considers the following graded rings: $$\displaystyle A(Z_{\wt M})=\bigoplus_{k\geq0} H^0(Z_{\wt M}, \pi^*\OO(k)),~ A(l_\infty)=\bigoplus_{k\geq0} H^0(l_\infty, \pi^*\OO(k)),$$
 and $I\subset A(Z_{\wt M})$  the ideal generated by  $u, v,$ the pull-backs of the two sections of $\OO(1)\to\bcp^1.$ 
The two sections $u$ and $v$ are associated to the homogeneous coordinates $[u:v]$ ofn$\bcp^1.$
 There is a relation between the graded rings given by \cite[Prop. 2.3]{kron2}:
 \begin{equation} \label{exact-seq}
 0\to I\to A(Z_{\wt M}) \to A(l_\infty)\to 0. 
 \end{equation}
If we look at the affine varieties associated to the coordinate rings, the above sequence 
 can be interpreted as saying that there exists a four-dimensional affine variety $Y$ and a fibration $\phi:Y\to \CC^2\ni(u,v),$ such that $\phi^{-1}(0,0)=\CC^2/\Gamma, $ i.e. :

 \begin{equation} \label{def}
 \begin{array}[c]{cccc}
Y&\stackrel{}{\supset}&\phi^{-1}(0,0)&=\CC^2/\Gamma\\
\downarrow\scriptstyle{\phi}&&\downarrow\scriptstyle{}\\
\CC^2&\stackrel{}{\ni}&(u,v)=(0,0)
\end{array}
\end{equation}


Moreover \cite{kron2}, $Y$ is a deformation of $\CC^2/\Gamma ~(\Gamma \subset SU(2)),$ and it can be considered  a subset of $\CC^5(x,y,z,u,v).$ On $Y$ there is an induced $\CC^*$ action given by the grading on the $A(\cdot)$ rings such that $\phi$ is $\CC^*-$equivariant, and the action of $\CC^*$ on $\CC^2$ is the usual multiplication. Let $\ol Z^s=(Y\setminus 0)/\CC^*$ and $Z^s=(Y\setminus l^s)/\CC^*,$ where $l^s=\{u=v=0\}\subset \CC^5.$ The map $\phi$ induces a projection $\pi^s: Z^s\to\bcp^1$. Then, $\ol Z^s$ and $Z^s$ are orbifold models of the twistor spaces $Z_{\ol {\wt M}}$ and $Z_{\wt M}.$ The singularities of $\ol Z^s$ and $Z^s$ are 
contained in fibers  ${(\pi^s)}^ {-1}(a),$ for finitely many points $a\in \bcp^1$ (\cite[Lemma 2.7]{kron2}), and they correspond to contracting the holomorphic $2-$spheres of self-intersection $(-2)$ in $(\wt M,J_a)\approx(\pi^{-1}(a)\subset Z_{\wt M}).$
The twistor space $Z_{\wt M}$ is obtained by taking the simultaneous minimal resolution $\chi: Z_{\wt M}\to Z^s, $ and we have the following commuting diagram:  
\begin{equation}\label{twist-id}
\begin{array}[c]{ccc}
Z_{\wt M}&\stackrel{\chi}{\rightarrow}&Z^s\\
\downarrow\scriptstyle{\pi}&&\downarrow\scriptstyle{\pi^s}\\
\bcp^1&\stackrel{=}{\rightarrow}&\bcp^1
\end{array}
\end{equation}
Moreover, $\chi$ extends to the compactifications of the twistor spaces $\ol \chi: Z_{\ol {\wt M}}\to \ol Z^s$ \cite{kron2}.

In the remaining part of the section we redirect our attention towards the study of the manifold $(M, J, g)$ by looking at the equivariant  objects associated to its universal cover  with induced complex structure and metric $(\wt M, \wt J, \wt g)$.
The group $G\cong \pi_1(M)$ acts freely on $(\wt M, \wt J, \wt g)$ by holomorphic isometries. An equivalent definition of the twistor space of a manifold is given by the space of unit self-dual forms $S(\Lambda^+)$ \cite{bes}. Thinking of $Z_{\wt M}$ in this way, we see that there is an induced free action of the group $G$ on the twistor space, and as the action of $G$ on $\wt M$ is by isometries, the induced action on $Z_{\wt M}$ is holomorphic. 
As the constructions are canonical,  the twistor space $Z_M$ is obtained as the quotient $Z_{\wt M}/G,$ and the same is true for their compactifications: $Z_{\ol M}=Z_{\ol {\wt M}}/G.$

As $\wt M$ is hyperk\"ahler there is a global trivialization of $S(\Lambda^+),$ and the $G$ will act on this bundle by fixing two points in each fiber, the self-dual forms associated to the complex structures $\wt J$ and $-\wt J.$ The action of $G$ on $S(\Lambda ^+)$ is non-trivial, as $G$ gives the extension of the group $\Gamma\subset SU(2)$ to $H\subset U(2)$. The projection $\pi: Z_{\wt M}\to \bcp^1$ is induced by this trivialization and, eventually after reparametrization of  bundle $S(\Lambda^+),$ we can assume that $G$ acts on the fibration $Z_{\wt M}\to \bcp^1$ such that the action on $\bcp^1$ has exactly two fixed points $[0:1]$ and $[1:0].$ Let's assume that the fiber $\pi^{-1}([0:1])$ can be identified to $(\wt M, \wt J).$ Moreover, as $G\cong\ZZ_n$ and the action of $G$ on $\bcp^1$ is holomorphic and non-trivial, it is  of the form $\rho[u:v]=[\rho^ku:v],$ where  $\rho \in G$ is thought of as a $n^{th}$ root of unity  and $k$ is some integer associated to the action $G.$ 
 We can conclude the following about the twistor space of $M$:

\begin{prop}
The twistor space $Z_M$ does not fiber over $\bcp^1,$ but instead there is a holomorphic fibration: $Z_M\to (\bcp^1/G).$ The generic fiber is diffeomorphic to $\wt M,$ while the fibers above the two singular points are $(M, \pm J)$.
\end{prop}

\begin{proof}[Proof of Theorem \ref{main-class}]
As we have seen in Section \ref{twistor}, Theorem \ref{ex-A}, Remark \ref{exist},  the simply connected case is completely understood from the work of Hitchin \cite{hit} and Kronheimer \cite{kron1,kron2}. We will finish the proof by treating the non-simply connected case.

%
%

 The map $\chi:Z_{\wt M}\to Z^s$ (\ref{twist-id}) is given by contracting the exceptional $(-2)-$spheres on each  fiber of  $\pi,$ or  equivalently looking at the fibration of the canonical models of the fibers. Hence the $G-$holomorphic action on $Z_{\wt M}$ induces a $G-$holomorphic action on $Z^s$ which is compatible with the fibration $\pi^s:Z^s\to\bcp^1, $ where $G$ acts on $\bcp^1$ as before, with fixed points $[0:1]$ and $[1:0]$. 
 
 The manifold $\wt M$ is mapped to ${(\pi^s)}^{-1}([0:1])\subset Z^s$ by $\chi.$   
As $\wt M$ is an ALE hyperk\"ahler manifold the only exceptional curves are a finite number of holomorphic spheres of self-intersection $-2,$ which are either isolated or form a tree configuration of type A, D or E. The group $G$ acts freely and holomorphically on $\wt M$ hence it is going to take holomorphic spheres and configurations of holomorphic spheres to them-selves. If $G$ or a subgroup of $G$ fixes such a configuration, then $G$ or the subgroup has to permute both the spheres and the intersection points of the spheres  of that configuration (as the action of $G$ is free). But as we have a tree configuration of spheres, there is one more sphere than intersection points, hence the subgroup of $G$ fixing a configuration of holomorphic spheres must be trivial. So, if we consider the manifold $\wt M^s\cong{(\pi^s)}^{-1}([0:1])\subset Z^s,$ the canonical model of $\wt M,$  then the induced action of $G$ on ${(\pi^s)}^{-1}([0:1])$  is free. The quotient $\wt M^s/G=:M^s$ is the canonical model of the manifold $(M, J).$

The fibration $\pi^s:Z^s\to \bcp^1$ is obtained by taking the $\CC^*-$quotient of $\phi:(Y\setminus l^s)\to \CC^2\setminus\{0\},$ where $\CC^*$ acts on $\CC^2$ by scalar multiplication. Hence, the $G-$action on $Z^s$ induces an action on $Y\setminus l^s.$ To see how this action acts on the central fiber $\phi^{-1}(0)\cong(\CC^2/\Gamma)$ we have to remember that the central fiber is associated to the orbifold compactification of the twistor space to $\ol{Z^s}=Z^s\cup (\bcp^1/\Gamma)=(Y\setminus\{0\})/\CC^*.$ But the compactification of  $Z_M$ is $Z_M\cup (l_\infty/G)=Z_M\cup(\bcp^1/H)$. Hence the action of $G=H/\Gamma$ on $\CC^2/\Gamma$ must be such that $(\CC^2/\Gamma)/G=(\CC^2/\Gamma)/(H/\Gamma)\cong\CC^2/H.$ 

Let $L$ be the line in $\CC^2$ corresponding to the point $[0:1]\in\bcp^1\cong(\CC^2\setminus\{0\})/\CC^*,$ and let 
  $\cal  X=\phi^{-1}(L)\subset Y$. As $Y$ is a deformation of a rational double point singularity, $\phi:\cal X\to L$ is a $1-$parameter Gorenstein deformation of $Y_0=\CC^2/\Gamma.$
 Moreover the action of $G$ restricts to $\cal X,$ it is compatible with the fibration, and $G$ acts trivially on $\CC=L.$ Hence we have $\phi:\cal X/G\to\CC$ is a $1-$parameter $\QQ-$Gorenstein deformation of $Y_0/G\cong\CC^2/H.$ 
We can use Theorem \ref{deform}, if $M$ does not contain exceptional $(-2)-$rational curves, or Proposition \ref{defman} otherwise, to
 conclude that $H$ is a finite cyclic group generated by a matrix of the form 
$$ \left(\begin{array}{cc}\rho & 0 \\0 & \rho^{dnm-1}\end{array}\right)$$
 with $\rho$ a root of unity of order $dn^2,$ where $n\geq2,$ and $m$  and $n$ are relatively prime. This also implies that the universal cover of $M$ has to be an $A-$type manifold, and the $\ZZ_n$ action is the one given in Proposition \ref{unsm}.

Moreover, $M$ is obtained from the $1-$parameter $\QQ-$Gorenstein deformation of $\CC^2/H$ given by $\phi:\cal X/G\to\CC,$ and if the fiber $M^s$ has singularities, they must be of the type $A$ and $M$ is the minimal resolution of $M^s$. 
The metric $g$ is the unique Ricci-flat K\"ahler metric in a given K\"ahler class.
\end{proof}

\section{Gibbons-Hawking formulation and ALF Ricci-flat K{\"a}hler metrics}\label{gib-hawk}

Hitchin's approach defines the metrics on $M$ via the twistor space.  An alternative construction of these metrics is due to  Eguchi, Hanson \cite{egha} on $A_1$ and Gibbons, Hawking \cite{giha} on any $A_k-$surfaces. It uses essentially the $S^1-$action on $(z_1, z_2)\in\CC^2$,  seen as a subgroup of $SU(2),$   or equivalently the representation of the angular coordinate of $A,$ as in section \ref{twistor}.

Let $x_i,$ $ i=1,\dots,k,$ be a collection of $k$ points in $\RR^3.$ Let 
  $V(x)=\frac12\sum_{i=1}^k \frac1{|x-x_i|}$ be a function defined on $\RR^3\s-\{x_i\}.$ It can be easily checked that $[\frac1{2\pi}*dV]\in H^2(\RR^3\s-\{x_i\},\ZZ),$ where $*$ denotes the Hodge star operator on the Euclidean $\RR^3.$ Associated to this cohomology class there is a principal $S^1$ bundle on $\RR^3\s-\{x_i\}.$  Let  $\pi':N'\to \RR^3\s-\{x_i\}$   be the total space of the bundle, with the natural projection. 
 Locally in a neighborhood of a point $x_i,$ the Chern class of the $S^1$ bundle is $-1,$ hence locally our bundle $N'$ is just the $S^1$ bundle corresponding to the Hopf fibration $\RR^4\s-\{0\}=\CC^2\s-\{0\}\to \RR^3\s-\{0\},$ where the $S^1$ action is multiplication by the unit complex numbers. We can then smoothly compactify $N'$ to $N=N'\cup\{\overline x_1,\dots ,\overline x_k\}$ by adding a point $\ol x_i$ corresponding to each $x_i.$ The $S^1$ action extends on $N$ by having the points $\{\overline x_1 ,\dots ,\overline x_k\}$ as its fixed points. Notice that the manifold $N$ has the same description as the smooth manifold $M_{re}$ associated to the deformation of a $A_k$ singularity and hence the same topology at infinity.
   
  Let $\omega$ be a connection $1-$form on $N'$ such that $ d\omega=\pi'^*(*dV).$  Let 
  $$g_0=\frac1V\omega^2+V\pi'^*ds^2$$ 
  be a metric on $N'$ where $ds^2$ is the Euclidean metric on $\RR^3.$ Then $g_0$ extends to a smooth ALE Ricci-flat metric $g$ on $N,$ which is known as the Gibbons-Hawking metric \cite{giha}. A different  choice of the connection $\omega$ will induce an isometric metric. 
  
  If we consider the differential manifold obtained by smoothing a $\frac1{dn^2}(1, dnm-1)$ singularity, we need to restrict ourselves to the case $k=dn,$ and choose the points $x_i$ to be the vertices of $d$ regular polygons, each with $n$ sides,  laying in horizontal planes $(b=ct)$ and centered at the origin.  
 Then the  manifold $N'$ admits an isometric action of the group $\ZZ_n$ analogous to the action  \ref{act-m^c}:
 \begin{equation}\label{act-gh}
\text{if}~\rho \in\CC, \rho^n=1,~~\text{then}~ \rho(a)=\rho^{-m}a,~\rho (b)=b, ~ \text{and} ~ \rho ( \theta)=\rho \theta \end{equation}
 where $(a,b)\in \CC\times\RR\cong\RR^3$ and $\theta\in S^1$ is a local coordinate of  the principle $S^1$ bundle, and on each fiber the  action of the group $\ZZ_n$ is just the action of the finite subgroup of order $n.$ This action extends to a free action on $N.$ The quotient manifold $N/\ZZ_n$ is diffeomorphic to the smoothing of a $\frac1{dn^2}(1, dnm-1)$ singularity. Moreover as both the function $V$ and the connection 1-form are group invariant, we have an induced ALE Ricci flat  metric on the quotient space. 
 On $N,$ \cite[Sec.2]{akl}, we can define a $\ZZ_n-$invariant complex structure as follows:  given $\frac{\partial}{\partial \theta}, \frac{\partial}{\partial b}, \frac{\partial}{\partial a_1}=\frac{\partial}{\partial Re(a)}, \frac{\partial}{\partial a_2}=\frac{\partial}{\partial Im(a)}$  the vector basis associated to the local coordinates $(\theta, b, a)\in N',$ the complex structure $J$ is defined as $J (V^{\frac12}\frac{\partial}{\partial \theta})=\frac{\partial}{\partial b}, J(V^{-\frac12}\frac{\partial}{\partial a_1})=V^{-\frac12}\frac{\partial}{\partial a_2}.$ An easy check shows that $J$ is  integrable, compatible with the metric $g_0$  and extends to the total space $N$, \cite{akl}. 
 
 As in the ALE-case, we can use the above description to construct  ALF (asymptotically locally flat) Ricci-flat  K\"ahler metrics on a manifold obtained by smoothing a $\frac1{dn^2}(1, dnm-1)$ singularity. A metric is said to be ALF if the complement of a compact subset is finitely covered by $(\RR^3\setminus Ball)\times S^1$ with the metric decaying to the flat product metric. These metrics are obtained by replacing the function $V(x)$, with $V'(x):=1+V(x).$  In this case, the above construction is going to give an ALF Ricci-flat  K\"ahler metric on the same smooth space $N$. These metrics are know as the multi-Taub-NUT metrics. As $V'(x)$ is $\ZZ_n$ invariant, we have an induced ALF Ricci-flat K\"ahler metric on $N/\ZZ_n.$ 
 
 In his paper \cite{min}, Minerbe gives a classification of gravitational instantons (hyperk\"ahler metrics) of cyclic type, showing that they are the flat $\CC\times \CC/\ZZ$ or the multi-Taub-NUT manifolds. If one wants to generalize his classification, to a classification of ALF Ricci-flat K\"ahler 4-manifolds, then the above quotients should be included.
 
 Another application of the Gibbons-Hawking construction is in the spirit of Anderson-Kronheimer-LeBrun's result \cite{akl}. We can choose infinitely many points $\{x_i\}_{i\in\NN}\subset\RR^3$ arranged in $n-$tuples which are vertices of infinitely many regular polygons $P_j, j\in \NN$. The polygons,  $P_j$ are chosen to   have all  $n$ sides, lie on horizontal planes (constant $b-$coordinates) and inscribed in circles of radii $j^2$ centered about the $b-$axis. Then, as in \cite{akl}, we have constructed a manifold $N_\infty$ with infinite homology. Moreover,  the metric and complex structure are invariant under a $\ZZ_n$ free action of the form   \ref{act-gh}. Hence, on the non-simply connected manifold $N_\infty/\ZZ_n$ there is an induced complete Ricci-flat K\"ahler metric. 
 
 The above examples of Ricci-flat K\"ahler manifolds which are ALF or have infinite topology lead us to conclude that the ALE condition on our classification is not only necessary for our proofs, but without it the classification is not known.
  
 \section{An application}
 
 The deformations of singularities in this paper give us  a {\it local model} for smoothing an orbifold and constructing good metrics on a nearby space. We would also like to understand if the same is true  for global deformations of compact orbifolds.
 In this context we can address the following problem: {\it If the orbifold manifold is endowed with a good orbifold K\"ahler metric, does a nearby smoothing admit a  good metric? } The author answers positively this question in the case of constant scalar curvature K\"ahler metric in \cite{suv} (in preparation). The more precise result is the following:

\begin{thm}\cite{suv} Let $(M_0,g_0)$ be a compact constant scalar curvature K{\"a}hler orbifold of dimension two with isolated singularities. Assume that there is no non-zero holomorphic vector field vanishing somewhere on $M_0$ and that $M_0$ admits a one-parameter $\QQ-$Gorenstein smoothing $\pi:{\cal M} \to \Delta \subset \CC$. Then any manifold $M_t=\pi^{-1}(t)$ sufficiently close to $M_0$ admits a constant scalar curvature K{\"a}hler metric.
\end{thm}


\section*{Acknowledgements} The author would like to thank  Professor Claude LeBrun for stimulating discussions on the subject.

\bibliographystyle{alpha} 

\end{document}